\documentclass[12pt]{article}
\title{Any Orthonormal Basis in High Dimension is Uniformly Distributed over the Sphere}
\author{
Sheldon Goldstein\footnote{Departments of Mathematics and Physics,
     Rutgers University, Hill Center, 
     110 Frelinghuysen Road, Piscataway, NJ 08854-8019, USA.
     E-mail: oldstein@math.rutgers.edu},
Joel L. Lebowitz\footnote{Departments of Mathematics and Physics,
     Rutgers University, Hill Center, 
     110 Frelinghuysen Road, Piscataway, NJ 08854-8019, USA.
     E-mail: lebowitz@math.rutgers.edu},\\
Roderich Tumulka\footnote{Department of Mathematics,
     Rutgers University, Hill Center, 
     110 Frelinghuysen Road, Piscataway, NJ 08854-8019, USA.
     E-mail: tumulka@math.rutgers.edu}, and
Nino Zangh\`\i\footnote{Dipartimento di Fisica, Universit\`a di
    Genova and INFN sezione di Genova, Via Dodecaneso 33, 16146
    Genova, Italy.  E-mail: zanghi@ge.infn.it}
}
\date{January 22, 2015}

\usepackage{amsthm,url,amssymb,amsmath,amsfonts,mathrsfs}
\usepackage[usenames,dvipsnames]{color}
\newcommand{\z}[1]{#1}

\theoremstyle{plain}
\newtheorem{thm}{Theorem}
\newtheorem{lemma}{Lemma}
\newtheorem{cor}{Corollary}
\theoremstyle{definition}\newtheorem{defn}{Definition}
\numberwithin{equation}{subsection}

\DeclareMathOperator{\tr}{tr}
\DeclareMathOperator{\Var}{Var}
\DeclareMathOperator{\Cov}{Cov}
\newcommand{\CCC}{\mathbb{C}}
\newcommand{\RRR}{\mathbb{R}}
\newcommand{\NNN}{\mathbb{N}}
\newcommand{\EEE}{\mathbb{E}}
\newcommand{\PPP}{\mathbb{P}}
\newcommand{\SSS}{\mathbb{S}}
\newcommand{\sphere}{\mathbb{S}}

\newcommand{\scp}[2]{\langle #1| #2 \rangle}
\newcommand{\Hilbert}{\mathscr{H}}

\newcommand{\be}{\begin{equation}}
\newcommand{\ee}{\end{equation}}
\newcommand{\vx}{\boldsymbol{x}}
\newcommand{\vy}{\boldsymbol{y}}
\newcommand{\vz}{\boldsymbol{z}}
\newcommand{\vX}{\boldsymbol{X}}
\newcommand{\vY}{\boldsymbol{Y}}

\newcommand{\vzbar}{\overline{\vz}}
\newcommand{\zbar}{\overline{z}}
\newcommand{\X}{\mathbb{X}^d}
\newcommand{\Y}{\mathbb{Y}}

\begin{document}
\maketitle
\begin{abstract}
Let $\X$ be a real or complex Hilbert space of finite but large dimension $d$, let $\SSS(\X)$ denote the unit sphere of $\X$, and let $u$ denote the normalized uniform measure on $\SSS(\X)$. For a finite subset $B$ of $\SSS(\X)$, we may test whether it is approximately uniformly distributed over the sphere by choosing a partition $A_1,\ldots,A_m$ of $\SSS(\X)$ and checking whether the fraction of points in $B$ that lie in $A_k$ is close to $u(A_k)$ for each $k=1,\ldots,m$. We show that if $B$ is any orthonormal basis of $\X$ and $m$ is not too large, then, if we randomize the test by applying a random rotation to the sets $A_1,\ldots,A_m$, 
$B$ will pass the random test with probability close to 1. 
This statement is related to, but not entailed by, the law of large numbers. An application of this fact in quantum statistical mechanics is briefly described.

\medskip

\end{abstract}


\section{Introduction}

Let $\X$ be a real or complex Hilbert space of finite but large dimension $d$, let $\SSS(\X)$ be the unit sphere in $\X$, and let $u=u_{\SSS(\X)}$ denote the uniform probability measure (i.e., normalized surface area) over $\SSS(\X)$. Given a large number of points on $\SSS(\X)$, we may ask whether these points are approximately uniformly distributed over $\SSS(\X)$. When we are given an orthonormal basis of $\X$, this provides us with $d$ points on $\SSS(\X)$, which may at first seem like too small a number, given that $\SSS(\X)$ has real dimension $d-1$ or $2d-1$, for rendering meaningful the question whether these points are approximately uniformly distributed. However, the question is meaningful in a suitably coarse-grained sense of ``approximately uniform,'' viz., in the sense that for a partition $A_1,\ldots,A_m$ of $\SSS(\X)$ with $m\ll d$,
the number of points in $A_k$, divided by $d$, is close to $u(A_k)$. 

One version of our result asserts
that for a \emph{random} orthonormal basis $\{b_1,\ldots,b_d\}$ with distribution $u_{ONB(\X)}$ (the uniform (normalized) measure over all orthonormal bases of $\X$, see below for more details), the empirical distribution on $\SSS(\X)$ of $b_1,\ldots,b_d$ is approximately uniform relative to the partition $A_1,\ldots,A_m$ with probability close to 1. Needless to say, for every fixed orthonormal basis $\{b_1,\ldots,b_d\}$ there exist partitions $A_1,\ldots, A_m$ for which the number of basis vectors in $A_k$, divided by $d$, is not at all close to $u(A_k)$; for example, $A_1=\{b_1,\ldots,b_d\}$ and $A_2=\SSS(\X)\setminus A_1$.

Our result as just formulated follows once we have it for $m=2$, i.e., for partitions consisting merely of a set $A$ and its complement. It therefore suffices to focus on the simpler statement that for any Borel set $A\subseteq\SSS(\X)$ and a $u_{ONB(\X)}$-distributed orthonormal basis,
\be\label{thm1eq1A}
\PPP\biggl(\frac{1}{d} \# \Bigl\{i\in\{1\ldots d\}: b_i\in A \Bigr\} \approx u(A) \biggr) \approx 1\,,
\ee
where $\# S$ denotes the number of elements of a finite set $S$.

Here is a different way of phrasing our result. A good concept of ``approximately uniformly distributed'' should be invariant under rotations (or unitary transformations) of $\X$; thus, if we claim of one orthonormal basis that it is approximately uniformly distributed, we should make this claim of \emph{every} orthonormal basis. So let us regard $\{b_1,\ldots,b_d\}$ as fixed and randomize $A$ instead by considering the uniform distribution over all sets $A'$ congruent to $A$. To this end, let $G$ be the orthogonal group or unitary group of $\X$, depending on whether $\X$ is real or complex, let $u_G$ be the normalized uniform measure (i.e., the Haar measure) over $G$, and let $R$ be a $u_G$-distributed random element of $G$. Our test set will be $A'=R(A)$ (i.e., a random rotation of $A$). Our result is that for every orthonormal basis $\{b_1,\ldots,b_d\}$ and every Borel set $A\subseteq\SSS(\X)$,
\be
\PPP\biggl(  \frac{1}{d} \# \Bigl\{i\in\{1\ldots d\}: b_i\in R(A) \Bigr\} \approx u(A) \biggr) \approx 1\,.
\ee

\subsection{Precise Formulation}

\begin{defn}
Let $\varepsilon,\delta>0$. We say of a finite set $B\subseteq\SSS(\X)$ that it is \emph{$\varepsilon$-$\delta$-uniform on $\SSS(\X)$} iff\footnote{This definition possesses a natural generalization to measures instead of finite sets: We say of a normalized measure $\mu$ on (the Borel $\sigma$-algebra of) $\SSS(\X)$ that it is \emph{$\varepsilon$-$\delta$-uniform on $\SSS(\X)$} iff for every Borel set $A\subseteq\SSS(\X)$, $\PPP\bigl(\bigl| \mu(R(A))-u(A) \bigr|\leq \delta \bigr)\geq 1-\varepsilon$. The definition for a finite set $B$ then corresponds to the measure $\mu =(\# B)^{-1}\sum_{b\in B} \delta_b$ with $\delta_b$ the point mass at $b$, i.e., $\mu(A')=\#(B\cap A')/\# B$.} for every Borel set $A\subseteq\SSS(\X)$,
\be\label{def1eq}
\PPP\biggl( \Bigl| \frac{\# (B\cap R(A))}{\# B} - u(A) \Bigr|\leq \delta \biggr) \geq1-\varepsilon\,.
\ee
\end{defn}

\begin{thm}\label{thm0}
(Version 1) For every $\varepsilon,\delta>0$ and every $d\geq 4$ with $d\geq \delta^{-2}\varepsilon^{-1}$,
every orthonormal basis $B$ in $\X=\RRR^d$ or $\CCC^d$ is $\varepsilon$-$\delta$-uniform on $\SSS(\X)$.
\end{thm}

Somewhat sloppily, we sometimes regard a basis as ordered (i.e., as a $d$-tuple) and sometimes as unordered (i.e., as a set). It does not matter which point of view we assume, and there are no bad consequences of switching the point of view; we call a $d$-tuple $\varepsilon$-$\delta$-uniform if the corresponding set (obtained by forgetting the order) is.

We use the notation $Y\sim \mu$ for saying that the random variable $Y$ has distribution $\mu$. For example, in \eqref{def1eq}, $R\sim u_G$.

Version 2 of Theorem~\ref{thm1} (see below) provides an alternative formulation in terms of a \emph{random} orthonormal basis.
The uniform distribution $u_{ONB(\X)}$ can be defined as the distribution of the random orthonormal basis $B$ obtained from a fixed orthonormal basis $B_0$ by applying a random rotation $R\sim u_G$, $B=R(B_0)$. The distribution of $B$ is, in fact, independent of the choice of $B_0$. Alternatively, a $u_{ONB(\X)}$-distributed basis $\{b_1,\ldots,b_d\}$ can be constructed as follows: Choose $b_1$ with distribution $u$ from $\SSS(\X)$; let $b_1^\perp$ denote the orthogonal complement of $b_1$ in $\X$, and $\SSS(b_1^\perp)$ the unit sphere in that subspace; choose $b_2$ uniformly in $\SSS(b_1^\perp)$; then choose $b_3$ uniformly in $\SSS(\{b_1,b_2\}^\perp)$; and so on. Theorem~\ref{thm0} can easily be seen to be equivalent to the following.

\setcounter{thm}{0}
\begin{thm}\label{thm1}
(Version 2) For every $\varepsilon,\delta>0$ and every $d\geq 4$ with $d\geq \delta^{-2}\varepsilon^{-1}$, the following is true: Let $\X=\RRR^d$ or $\CCC^d$ and $B\sim u_{ONB(\X)}$. For every Borel set $A\subseteq \sphere(\X)$,
\be\label{thm1eq}
\PPP\biggl(\Bigl|\frac{ \# (B\cap A)}{d} - u(A)\Bigr| \leq \delta \biggr) \geq 1-\varepsilon\,.
\ee
\end{thm}

\begin{cor}\label{cor1}
For every $\varepsilon,\delta>0$, every $m\in\NNN$, and every $d\geq 4$ with $d\geq m \delta^{-2}\varepsilon^{-1}$, the following is true: Let $\X=\RRR^d$ or $\CCC^d$ and $B\sim u_{ONB(\X)}$. For every partition $A_1,\ldots, A_m$ of $\sphere(\X)$ consisting of Borel sets, 
\be\label{cor1eq}
\PPP\biggl(\forall k\in\{1\ldots m\}:\Bigl|\frac{ \# (B\cap A_k)}{d} -u(A_k)\Bigr|\leq \delta \biggr) \geq 1-\varepsilon\,.
\ee
\end{cor}

In the \emph{real} case $\X=\RRR^d$, the statements refer to orthonormal bases of both orientations (``left-handed'' and ``right-handed''). Theorem~\ref{thm1} and Corollary~\ref{cor1} remain true when restricted to just one orientation, provided $d\geq 2m\delta^{-2}\varepsilon^{-1}$.
After all, if $99\%$ of all orthonormal bases have a property $p$ and $50\%$ of all orthonormal bases are left-handed, then at least $98\%$ of all left-handed orthonormal bases must have the property $p$.

It is perhaps not surprising that the basis vectors are uniformly distributed, as their orthogonality will have the ``repulsive'' effect that no two of them can be close to each other. On the other hand, one might have expected that in order to obtain a uniformly distributed set, one has to use $\{\pm b_1,\ldots,\pm b_d\}$, while the basis vectors $\{b_1,\ldots,b_d\}$ alone tend (one might have expected) to clump on one side of the sphere, as they all lie on a cone around $b_1+\ldots +b_d$; however, when $d$ is large then the opening angle of this cone, $2\arccos(d^{-1/2})$, is approximately $\pi-2/\sqrt{d}$ and thus close to $\pi$ (or $180^\circ$), so not very clumped after all. 

The following version of Theorem 1 expresses the theorem in terms of test functions $\varphi:\SSS(\X)\to \RRR$ rather than test sets $A\subseteq\SSS(\X)$. Let $\EEE_\mu (f)$ and $\Var_\mu(f)$ denote the mean and variance of the function $f:\Omega\to \RRR$ relative to the probability measure $\mu$ on $\Omega$.

\setcounter{thm}{0}
\begin{thm}
(Version 3) For every $\varepsilon,\delta>0$, every $d\geq 4$ with $d\geq 2\delta^{-2}\varepsilon^{-1}$, every orthonormal basis $\{b_1,\ldots,b_d\}$ of $\X=\RRR^d$ or $\CCC^d$, $R\sim u_G$, and every test function $\varphi\in L^2(\SSS(\X),u,\RRR)$, 
\be\label{thm1eqvarphi}
\PPP\biggl( \Bigl| \frac{1}{d} \sum_{j=1}^d \varphi(R (b_j)) - 
\EEE_u(\varphi)\Bigr|\leq \delta \sqrt{\Var_u(\varphi)}\biggr) \geq1-\varepsilon\,.
\ee
\end{thm}

With the same methods as in this paper, one can perhaps show also that, in high dimension $d$, the action of the rotation (resp., unitary) group $G$ on the unit sphere is weakly mixing, i.e., that for $R\sim u_G$ and any two measurable sets $A,B\subseteq \SSS(\X)$, $\PPP\bigl( u(A\cap R(B))\approx u(A) u(B) \bigr) \approx 1$.

A physical application of our result is outlined in Section~\ref{sec:appl}.

\subsection{Comparison to Known Results}

Theorem~\ref{thm1} can be regarded as a \emph{typicality theorem}, i.e., as a statement about the \emph{typical} behavior of something, here of orthonormal bases or sets congruent to a given set $A\subseteq\SSS(\X)$. Well-known examples of typicality theorems about spheres in high dimension include the following statements: (i) that in high dimension, most of the area of a sphere is near the equator, (ii) that in high dimension, most of the volume of the unit ball is near the surface.

Theorem~\ref{thm1} is similar to an instance of the \emph{law of large numbers}, i.e., of the statement that if $X_1,\ldots,X_n$ are independent identically distributed (i.i.d.)\ random variables then for sufficiently large $n$ their empirical distribution is arbitrarily close to the distribution of $X_1$ with probability arbitrarily close to 1. Suppose $b_1,\ldots,b_d$ were \emph{independent} $u$-distributed random vectors on $\SSS(\X)$, and let $K_i$ be (in every realization) 1 or 0 depending on whether $b_i\in A$ or not. 
Then the $K_i$ are i.i.d.\ random variables with distribution $\PPP(K_i=1)=u(A)$, $\PPP(K_i=0)=1-u(A)$, and the law of large numbers implies that \eqref{thm1eq} holds for sufficiently large $d$. 
Now in the situation of Version 2 of Theorem~\ref{thm1}, $b_1,\ldots,b_d$ are \emph{not} independent (since they have to be exactly orthogonal to each other), but they are approximately independent in the following sense: if we pick two independent random (uniformly distributed) vectors $x,y$ on $\SSS(\X)$ with large $d$, then they are anyhow, with high probability, approximately orthogonal. (Indeed, it follows from symmetry considerations that their inner product $\scp{x}{y}$ has expectation $\EEE\scp{x}{y}=0$ and variance $\EEE |\scp{x}{y}|^2=1/d$, so $\scp{x}{y}$ will typically be small like $1/\sqrt{d}$.) So Version 2 of Theorem~\ref{thm1} can be regarded as saying that the weak dependence between the basis vectors $b_i$ does not disturb the relation \eqref{thm1eq} provided by the law of large numbers.

Known theorems about \emph{uniformity} (or \emph{equidistribution}) are usually rather different in character from our result. One type of theorem asserts that a certain sequence $x_n$ of points (e.g., $x_n=n\alpha$ mod 1 for irrational $\alpha$) is uniformly distributed as $n\to\infty$ in some set (e.g., the unit interval) \cite{W16}; another type concerns how uniformly certain paths (e.g., billiard trajectories) fill the space they are in \cite{Beck2,Beck}. Another circle of questions closer to our result, described in \cite[Sec.~2]{Beck}, concerns quantifying how uniformly distributed a set $\{x_1,\ldots,x_n\}$ in (say) the unit interval $[0,1]$ is by comparing, for some test function $\varphi:[0,1]\to\RRR$, $n^{-1}\sum_{i=1}^n \varphi(x_i)$ to $\int_0^1 dx\,\varphi(x)$. If the $x_i$ are chosen at random (independently with uniform distribution), then the difference is (with high probability) of order $n^{-1/2}$ for any $\varphi\in L^2([0,1])$. However, if the $x_i$ are evenly spaced, $x_i=i/n$, and $\varphi$ is sufficiently smooth, then the difference is of order $n^{-1}$ or even smaller (see \cite[Sec.~2]{Beck} for a discussion). Thus, in a certain sense, some sets $\{x_1,\ldots,x_n\}$ are \emph{very} uniform. Our result can be expressed by saying that, for a random orthonormal basis $\{b_1,\ldots,b_d\}$, $d^{-1}\sum_{i=1}^d \varphi(b_i)$ is (with high probability) close to the mean of $\varphi$ for any $\varphi\in L^2(\SSS(\X))$ with typical error of order at most $d^{-1/2}$, see \eqref{thm1eqvarphi}; we leave open the question whether, for sufficiently smooth functions, the error is smaller than that.

Further facts that are somewhat related to our result
come from the field of \emph{geometric probability}. Wendel \cite{Wen62} considered $X_1,\ldots,X_n$ independent $u$-distributed on $\SSS(\RRR^d)$ and computed the probability that there exists a hemisphere containing all $n$ points. A result described in \cite[p.~326]{San76} concerns random rotations $R_1,\ldots,R_n$ in $\X=\RRR^3$ that are independent $u_G$-distributed and provides a formula, for arbitrary convex sets $A,A'\subseteq \SSS(\RRR^3)$, for the probability that $A'\cap R_1 (A)\cap \ldots \cap R_n(A) \neq \emptyset$. Further similar results (and open problems) are described in \cite{San76,Sol78}.

The phenomenon of \emph{concentration of measure} \cite{MS86,Led01}, which can occur in a space $\Y$ equipped with both a metric and a measure, refers to the situation that most points $y\in\Y$ (in terms of the measure) are close (in terms of the metric) to a certain set that is small in terms of the measure. For example, this occurs for $\Y=\SSS(\RRR^d)$ in high dimension $d$, where most points are close to the equator. As a consequence known as Levy's lemma \cite[p.~6]{MS86}, every 1-Lipschitz function $f:\SSS(\RRR^d)\to \RRR$ (i.e., with $|f(\vx)-f(\vy)|\leq \text{distance}(\vx,\vy)$) is almost constant, i.e., is close to its median (or mean, for that matter) at most points. Theorem~\ref{thm1} is somewhat similar, as it asserts (in Version 2) that the function $f_A$ on the set $ONB(\X)$ of orthonormal bases of $\X$ defined by $f_A(B)=d^{-1}\#(B\cap A)$ is almost constant for every $A$. (More generally, every function $f$ on $ONB(\X)$ of the form $f(b_1,\ldots,b_d)=\sum_{i=1}^d \varphi(b_i)$ with $\varphi\in L^2(\SSS(\X))$ is almost constant, in the sense expressed in Version 3 of Theorem~\ref{thm1}.)

\z{A fact related to Theorem 1 and concentration of measure is \emph{Raz's lemma} \cite{Raz99,MW03,KR11}, which roughly asserts the following: Let $A$ be a subset of $\SSS(\RRR^d)$, and let $1\ll k<d$. For most $k$-dimensional subspaces $U\subseteq \RRR^d$, $u_U(A \cap U) \approx u(A)$, where $u_U$ is the normalized uniform measure on $\SSS(U)$.}

\subsection{Ideas of Proof}

Our proof of Theorem~\ref{thm1} is based on Theorem~\ref{thm2} below. Let $\Var(Y)$ (and $\Cov(X,Y)$) denote the variance (covariance) of the random variable $Y$ (variables $X$ and $Y$).

\begin{thm}\label{thm2}
Let $d\geq 4$, let $\X=\RRR^d$ or $\X=\CCC^d$, let $\{b_1,\ldots,b_d\}\sim u_{ONB(\X)}$, and let 
$\varphi\in L^2(\SSS(\X),u,\RRR)$. Then
\be\label{CovdVar}
   \Bigl| \Cov\bigl(\varphi(b_1),\varphi(b_2)\bigr) \Bigr| \leq \frac{1}{d-1} \Var_u(\varphi)\,.
\ee
The estimate is sharp in the sense that for every $d$ there exists a $\varphi$ for which equality holds.
\end{thm}

Theorem~\ref{thm2} expresses the fact that the $b_i$ are \emph{weakly} correlated. If they were independent, the covariance of $\varphi(b_1)$ and $\varphi(b_2)$ would be zero; since each $b_i$ has distribution $u$, $\Var(\varphi(b_i))=\Var_u(\varphi)$, and \eqref{CovdVar} states that the correlation coefficient of $\varphi(b_1)$ and $\varphi(b_2)$ is small (viz., no greater than $1/(d-1)$).

The proof of Theorem~\ref{thm1} (say, in Version 3) proceeds by 
noting that the random quantity $d^{-1}\sum_{i=1}^d \varphi(R(b_i))$ has expectation equal to the mean of $\varphi$ and showing that it has small variance. The variance of a sum $\sum \varphi(R(b_i))$ is the sum of the variances $\Var(\varphi(R(b_i)))$ plus the sum of the covariances $\Cov(\varphi(R(b_i)),\varphi(R(b_j)))$ for $i\neq j$; the variances can be computed, and the covariances can be estimated using Theorem~\ref{thm2}. Chebyshev's inequality then yields Theorem~\ref{thm1}.

The proof of Theorem~\ref{thm2} is, in turn, based on Theorem~\ref{thm3} below. Let $\vx^\perp$ denote the subspace orthogonal to $\vx\in\X$,
\be\label{xperpdef}
\vx^\perp=\{\vy\in \X: \scp{\vx}{\vy}=0\}\,,
\ee
$\SSS(\vx^\perp)$ the unit sphere in that subspace, and $u_{\SSS(\vx^\perp)}$ the normalized uniform measure over that sphere. In the following, we use the double factorial notation 
\be\label{n!!def}
n!!= \begin{cases} 1\cdot 3 \cdot 5 \cdots (n-2)\cdot n & \text{if $n$ is odd}\\
2\cdot 4 \cdot 6 \cdots (n-2) \cdot n & \text{if $n$ is even,} \end{cases}
\ee
and $0!!=1$.

\begin{thm}\label{thm3}
Suppose $d\geq 4$ and, again, $\X=\RRR^d$ or $\X=\CCC^d$. The equation
\be\label{Tdef}
(T \psi)(\vx) = \int\limits_{\sphere(\vx^\perp)} \!\! u_{\SSS(\vx^\perp)}(d\vy)\, \psi(\vy) \,,
\ee
defines a bounded, self-adjoint operator $T:\Hilbert\to\Hilbert$ on the ($\infty$-dimensional, complex) Hilbert space $\Hilbert=L^2(\SSS(\X),u,\CCC)$. Furthermore, $T$ has pure point spectrum, and its eigenvalues are: for $\X=\RRR^d$,
\be
0,\: 1, \text{ and }(-1)^{\ell/2} \frac{(\ell-1)!!(d-3)!!}{(\ell+d-3)!!} \text{ for }\ell=2,4,6,\ldots\:,
\ee
and for $\X=\CCC^d$,
\be
0,\: 1, \text{ and }(-1)^\ell \binom{\ell+d-2}{\ell}^{-1} \text{ for }\ell=1,2,3,\ldots\:.
\ee
For both $\X=\RRR^d$ or $\X=\CCC^d$, the largest absolute eigenvalue of $T$ is 1, with a 1-dimensional eigenspace formed by the constant functions, and the second largest absolute eigenvalue of $T$ is $1/(d-1)$.
\end{thm}

\z{The operator $T$ is related to the Radon transformation, the differences being that one integrates only over the unit sphere, and that the only hyperplanes considered are those passing through the origin. In \cite{KR11}, this operator is called the \emph{spherical Radon transformation}. 

A result very similar to Theorem~\ref{thm2} is Theorem 5.2 in \cite{KR11}, which, however, neither implies nor is implied by our Theorem~\ref{thm2}. To facilitate the comparison, we can express \eqref{CovdVar} in terms of the operator $T$ introduced in \eqref{Tdef} as
\be
\Bigl|\scp{\varphi}{T\varphi}-({\textstyle\int} \varphi)^2\Bigr|
\leq \frac{1}{d-1} \, \Bigl\|\varphi-{\textstyle\int}\varphi\Bigr\|_{L^2}^2
\ee
(with $\int\varphi = \int_{\SSS(\X)} u(d\vx) \, \varphi(\vx)$), and in fact it follows from Theorem~\ref{thm3} that
\be\label{chiTphi}
\Bigl|\scp{\chi}{T\varphi}-({\textstyle\int}\chi)({\textstyle\int}\varphi)\Bigr|
\leq \frac{1}{d-1} \, \Bigl\|\chi-{\textstyle\int}\chi\Bigr\|_{L^2} \,\Bigl\|\varphi-{\textstyle\int}\varphi\Bigr\|_{L^2}\,.
\ee
Theorem 5.2 in \cite{KR11} provides a bound for the left-hand side of \eqref{chiTphi} in terms of the $L^\infty$ norms of $\chi$ and $\varphi$ in case these norms are not too large.}

In Section~\ref{sec:proofs}, we provide proofs of Theorems~\ref{thm0}--\ref{thm3} and Corollary~\ref{cor1}; our proofs make repeated use of the rotational/unitary symmetry of the problem. In Section~\ref{sec:appl}, we briefly outline a physical application discussed in detail in \cite{GLMTZ14}.

\section{Proofs}
\label{sec:proofs}

\subsection{Proof of Theorems~\ref{thm0}--\ref{thm2} and Corollary~\ref{cor1} From Theorem~\ref{thm3}}

\begin{proof}[Proof of Version 1 of Theorem~\ref{thm0} from Version 2]
Suppose Version 2 is true. Fix $\varepsilon,\delta>0$, let $d\geq 4$ and $d\geq \delta^{-2}\varepsilon^{-1}$, and let $B_0$ be any fixed orthonormal basis in $\X=\RRR^d$ or $\CCC^d$. Then a random orthonormal basis $B$ with distribution $u_{ONB(\X)}$ can be thought of as obtained from $B_0$ by applying a random rotation, $B=R^{-1}B_0$ with $R\sim u_G$ (which implies $R^{-1}\sim u_G$). Then
\be
\#(B_0\cap R(A))=\#\bigl(R^{-1}(B_0)\cap A\bigr) = \#(B\cap A)\,,
\ee
so \eqref{def1eq} is equivalent to \eqref{thm1eq}.
\end{proof}

\begin{proof}[Proof of Corollary~\ref{cor1} from Version 2 of Theorem~\ref{thm1}]
Let $d\geq 4$, and let $E_{k,d}$ denote the event that
\be\label{Ekdef}
\Bigl| \frac{\# (B\cap A_k)}{d} - u(A_k) \Bigr| \leq \delta\,.
\ee
Version 2 of Theorem~\ref{thm1}, with $\varepsilon$ replaced by $\varepsilon/m$, yields that  for any $d\geq m\delta^{-2}\varepsilon^{-1}$, $\PPP(E_{k,d})\geq1-\varepsilon/m$, and thus $\PPP(E_{1,d}\cap \ldots \cap E_{m,d})\geq  1-\varepsilon$. 
\end{proof}

\begin{proof}[Proof of Version 2 of Theorem~\ref{thm0} from Version 3]
Let $\varphi$ be the indicator function of $A$. Then $\varphi$ lies in $L^2(\SSS(\X),u,\RRR)$, has mean $u(A)$ and variance
\be
\Var_u(\varphi) = u(A)(1-u(A))\leq \frac{1}{4}\,.
\ee
If we think of $B\sim u_{ONB(\X)}$ again as obtained by applying a random rotation $R\sim u_G$ to a fixed orthonormal basis $B_0=\{b_1,\ldots,b_d\}$, then
\be
\#(B\cap A) = \#(R(B_0)\cap A)= \sum_{j=1}^d \varphi(R(b_j))\,.
\ee
Thus, if we replace $\delta$ in Version 3 by $2\delta$, we obtain that Version 2 is true for $d\geq \frac{1}{2} \delta^{-2}\varepsilon^{-1}$, and in particular for $d\geq \delta^{-2}\varepsilon^{-1}$. (We dropped the factor $\frac{1}{2}$ in Version 2 for the sake of simplicity.)
\end{proof}

\begin{proof}[Proof of Version 3 of Theorem~\ref{thm1} from Theorem~\ref{thm2}]
Let $\varepsilon,\delta>0$, and let $R_j:= R(b_j)$, so that $\{R_1,\ldots,R_d\}\sim u_{ONB(\X)}$. Then 
\be
f(R):= \frac{1}{d}\sum_{j=1}^d \varphi(R_j)
\ee
has mean (since each $R_j$ is $u$-distributed)
\be
\EEE f(R)
=\frac{1}{d} \sum_{i=1}^d \EEE \varphi(R_j) 
= \EEE_u(\varphi)
\ee
and variance (since the $R_j$ are exchangeable)
\begin{align}
\Var(f(R)) 
&=\frac{1}{d^2} \Bigl(\sum_{i=1}^d \Var(\varphi(R_i)) + \sum_{\substack{i,j=1\\ i\neq j}}^d \Cov\bigl(\varphi(R_i),\varphi(R_j)\bigr) \Bigr)\\
&=\frac{1}{d^2} \Bigl(d \Var(\varphi(R_1)) + (d^2-d) \Cov\bigl(\varphi(R_1),\varphi(R_2)\bigr) \Bigr)\\
&\leq\frac{1}{d} \Var_u(\varphi) + \frac{d-1}{d} \Bigl|\Cov\bigl(\varphi(R_1),\varphi(R_2)\bigr) \Bigr|\\
&\leq\frac{2}{d} \Var_u(\varphi)
\end{align}
by \eqref{CovdVar} of Theorem~\ref{thm2} for $d\geq 4$.
Chebyshev's inequality (see, e.g., \cite[p.~65]{Bill}) asserts that for any random variable $X$,
\begin{equation}\label{Chebyshev}
\PPP\Bigl( \bigl| X - \EEE X \bigr| \geq \eta \Bigr) 
\leq \frac{1}{\eta^2} \Var(X)\,.
\end{equation}
Setting $X=f(R)$ and $\eta=\delta\sqrt{\Var_u(\varphi)}$, we obtain that
\begin{align}
\PPP\Bigl( \bigl| f(R) - \EEE_u(\varphi) \bigr| \geq \delta\sqrt{\Var_u(\varphi)} \Bigr) 
&\leq \frac{1}{\delta^2\Var_u(\varphi)} \Var(f(R))\\
&\leq \frac{2}{\delta^2 d}  \,,\label{error}
\end{align}
which yields \eqref{thm1eqvarphi} if $2/(\delta^2 d)\leq \varepsilon$, thus proving Version 3 of Theorem~\ref{thm1}. 
\end{proof}

\begin{proof}[Proof of Theorem~\ref{thm2} from Theorem~\ref{thm3}]
Since \eqref{CovdVar}, when true, will remain true if $\varphi$ is changed by adding a constant, we can assume without loss of generality that $\varphi$ has mean 0. 
Thus,
\be
\Var_u(\varphi) = \int\limits_{\SSS(\X)}u(d\vx) \, |\varphi(\vx)|^2 = \|\varphi\|^2_{\Hilbert}\,.
\ee
Let $|1\rangle$ denote the constant 1 function in $\Hilbert$. The property of mean 0 can be expressed as $\scp{1}{\varphi}=0$, or $\varphi\in |1\rangle^\perp$.

We can think of the joint distribution of $b_1$ and $b_2$ as follows: $b_1$ is chosen uniformly on $\SSS(\X)$, then $b_2$ is chosen uniformly on $\SSS(b_1^\perp)$. Thus,
\begin{align}
\Cov\bigl(\varphi(b_1),\varphi(b_2)\bigr) 
&= \int\limits_{\SSS(\X)}u(d\vx) \int\limits_{\SSS(\vx^\perp)}u_{\SSS(\vx^\perp)}(d\vy)\, \varphi(\vx)\, \varphi(\vy)\\
&= \scp{\varphi}{T\varphi}
\end{align}
by Theorem~\ref{thm3}, with $\scp{\cdot}{\cdot}$ the inner product in $\Hilbert$. Since, by Theorem~\ref{thm3}, $T$ is self-adjoint, and $\CCC|1\rangle$ is the eigenspace of $T$ with eigenvalue 1, its orthogonal complement $|1\rangle^\perp$ is mapped by $T$ to itself. Since $\varphi$ lies in $|1\rangle^\perp$, we have by the Cauchy--Schwarz inequality that
\begin{align}
\bigl|\scp{\varphi}{T\varphi}\bigr|& \leq \|\varphi\| \, \|T\varphi\| \label{T1}\\
&\leq \|T\|_{|1\rangle^\perp} \, \|\varphi\|^2 \label{T2}
\end{align}
with $\|T\|_{|1\rangle^\perp}$ the operator norm of $T$ on $|1\rangle^\perp$. By Theorem~\ref{thm3} again, $T$ has pure point spectrum, so the operator norm of $T$ on $|1\rangle^\perp$ is the largest absolute non-1 eigenvalue, which is $1/(d-1)$. Equality holds in \eqref{T1} and \eqref{T2} when $\varphi$ is an associated eigenfunction. 
We thus have \eqref{CovdVar} for $d\geq 4$, including the statement that equality holds in \eqref{CovdVar} for suitable $\varphi$, viz., for the eigenfunction with absolute eigenvalue $1/(d-1)$. 
\end{proof}

\subsection{Proof of Theorem~\ref{thm3} in the Real Case}

For the proof of Theorem~\ref{thm3}, we need Lemma~\ref{lemma:avg} below, for which we offer two different proofs, the first of which is based on Lemma~\ref{lemma:invariant} below. Let $S_{\ell}$ denote the group of permutations of $\{1,\ldots,\ell\}$.

\begin{lemma}\label{lemma:invariant}
Let $d\geq 3$ and $\ell\geq 1$.
Suppose the rank-$\ell$ tensor $A\in (\RRR^d)^{\otimes \ell}$ is symmetric,
\be\label{sigmaA}
A_{i_1\ldots i_\ell}=A_{i_{\sigma(1)}\ldots i_{\sigma(\ell)}}
\qquad \forall \sigma \in S_\ell\,,
\ee
and invariant under (the obvious action of) the orthogonal group $O(d)$,
\be\label{MA}
\sum_{j_1\ldots j_\ell=1}^d  M_{i_1j_1}\cdots M_{i_\ell j_\ell}A_{j_1\ldots j_\ell} = A_{i_1\ldots i_\ell}
\qquad \forall M\in O(d)\,.
\ee
If $\ell$ is odd then $A=0$, and if $\ell$ is even then $A$ is a multiple of $\tilde{A}$ given by the symmetrization of $\delta_{i_1i_2}\delta_{i_3i_4}\cdots \delta_{i_{\ell-1}i_\ell}$,
\be
\tilde{A}_{i_1\ldots i_\ell} = \frac{1}{\ell!} \sum_{\sigma\in S_\ell} \delta_{i_{\sigma(1)} i_{\sigma(2)}} \delta_{i_{\sigma(3)} i_{\sigma(4)}} \cdots \delta_{i_{\sigma(\ell-1)} i_{\sigma(\ell)}}\,,
\ee
where $\delta_{ij}$ is the Kronecker symbol (unit matrix).
\end{lemma}

\begin{proof}[Proof of Lemma~\ref{lemma:invariant}]
The lemma can be translated into a statement about homogeneous polynomials $P(x_1,\ldots,x_d)$ of degree $\ell$; the translation is based on writing such a polynomial in the form
\be\label{PA}
P(x_1,\ldots,x_d) = \sum_{i_1,\ldots,i_\ell=1}^d A_{i_1\ldots i_\ell}\, x_{i_1}\cdots x_{i_\ell}\,,
\ee
where the coefficients $A_{i_1\ldots i_\ell}$ can be taken to be symmetric under permutation of the indices. Lemma~\ref{lemma:invariant} is thus equivalent to the following: 

\textit{Suppose the homogeneous polynomial $P(x_1,\ldots,x_d)$ of degree\footnote{Although a polynomial of degree $\ell$ is usually taken to be non-zero, we include here the possibility $P=0$.} $\ell$ is $O(d)$-invariant. If $\ell$ is odd then $P=0$, and if $\ell$ is even then $P$ is a multiple of $(x_1^2+ \ldots + x_d^2)^{\ell/2}$.}

To see this, note that since $P$ is $O(d)$-invariant, its restriction to $\SSS(\RRR^d)$ must be constant. Since $P$ is homogeneous of degree $\ell$, it must be of the form $c|\vx|^2$, where $|\vx|=\sqrt{x_1^2+\ldots+x_d^2}$, and $c$ is a constant. If $\ell$ is odd, then invariance under the matrix $M\in O(d)$ with entries $M_{ij}=-\delta_{ij}$ implies that $P=-P$, so $P=0$.
\end{proof}

\begin{lemma}\label{lemma:avg}
Suppose $d\geq 3$ and $\ell\geq 1$. Let $P$ be a homogeneous real polynomial of degree $\ell$ in $d$ variables, 
\be
P(x_1,\ldots,x_d) = \sum_{i_1\ldots i_\ell=1}^d C_{i_1\ldots i_\ell} \, x_{i_1}\cdots x_{i_\ell}
\ee
with a symmetric tensor $C$.
Then the average of $P$ over the unit sphere is
\be\label{Pavg}
\int\limits_{\SSS(\RRR^d)} \!\! u(d\vx)\, P(\vx) = 
\begin{cases}
0 & \text{if $\ell$ odd}\\[3mm]
\alpha_{\ell,d} \sum\limits_{i_1\ldots i_{\ell/2}=1}^d C_{i_1i_1i_2i_2\ldots i_{\ell/2}i_{\ell/2}} & \text{if $\ell$ even}
\end{cases}
\ee
with
\be\label{alpha}
\alpha_{\ell,d} = \frac{(\ell-1)!!(d-2)!!}{(\ell+d-2)!!}\,.
\ee
\end{lemma}

\begin{proof}[First proof of Lemma~\ref{lemma:avg}]
By linearity, the average of $P$ must be
\be\label{AC}
\sum_{i_1\ldots i_\ell=1}^d A_{i_1\ldots i_\ell} \, C_{i_1\ldots i_\ell}
\ee
with
\be\label{Adef}
A_{i_1\ldots i_\ell} = \int\limits_{\SSS(\RRR^d)} \!\! u(d\vx) \, x_{i_1}\cdots x_{i_\ell}\,.
\ee
The tensor $A$ is symmetric and $O(d)$-invariant. 
By Lemma~\ref{lemma:invariant}, $A=0$ for odd $\ell$ and
\be\label{AalphatildeA}
A=\alpha_{\ell,d} \tilde{A}
\ee
for even $\ell$, with $\tilde{A}$ as in Lemma~\ref{lemma:invariant} and some constant $\alpha_{\ell,d}$. Thus, for even $\ell$, the average of $P$ is given by \eqref{AC} with $A$ replaced by $\alpha_{\ell,d}\tilde{A}$. Since $C$ is symmetric, this value is equal to
\be
\alpha_{\ell,d} \sum_{i_1\ldots i_\ell=1}^d \delta_{i_1i_2}\delta_{i_3i_4}\cdots \delta_{i_{\ell-1}i_\ell} \, C_{i_1\ldots i_\ell}\,.
\ee
(That is, it is not necessary to symmetrize the product of the $\delta$s, since $C$ is symmetric.) This proves \eqref{Pavg}.

To compute $\alpha_{\ell,d}$ for even $\ell$, it suffices to compare one nonzero component of $A$ and $\tilde{A}$, say $A_{11\ldots 1}$ (the average of $x_1^\ell$) and $\tilde{A}_{11\ldots 1}=1$. To compute $A_{11\ldots 1}$, we use spherical coordinates (with $r=1$), setting $x_1=\cos \theta$. 
Let 
\be\label{cdef}
c(d)=
\begin{cases}
   2&\text{if $d$ odd}\\
   \pi&\text{if $d$ even} 
\end{cases}
\ee
and
\be
g(n) = \prod_{k=1}^n c(k) = \underbrace{\: 2\pi2\pi2\cdots\:}_{\text{$n$ factors of 2 or $\pi$}}=
\begin{cases}
   2^{(n+1)/2}\pi^{(n-1)/2} & \text{if $n$ odd}\\
   (2\pi)^{n/2} & \text{if $n$ even.} 
\end{cases}
\ee
We note \cite{Sn} that for $n\geq 2$, we have
\be
|\SSS(\RRR^n)|=\frac{n\pi^{n/2}}{\Gamma(n/2+1)}
= \frac{g(n)}{(n-2)!!}
\ee
for the surface area of $\SSS(\RRR^n)$. Thus, for 
even $\ell$,
\begin{align}
\alpha_{\ell,d} 
&=\int_{\SSS(\RRR^d)} x_1^{\ell} \, u(d\vx) \\
&=\frac{1}{|\SSS(\RRR^d)|} \int_0^\pi d\theta\, \cos^\ell \theta \, \sin^{d-2}\theta \, |\SSS(\RRR^{d-1})| \\
&=\frac{(d-2)!!}{c(d)\,(d-3)!!}c(d)\frac{(\ell-1)!!(d-3)!!}{(\ell+d-2)!!}\\
&=\frac{(\ell-1)!!(d-2)!!}{(\ell+d-2)!!}
\end{align}
using
\be
\int_0^\pi d\theta\, \sin^p \theta \cos^q \theta =
  c(p)\frac{(q-1)!!(p-1)!!}{(p+q)!!}
\ee
if $q$ is even. 
This proves \eqref{alpha}.
\end{proof}

\begin{proof}[Second proof of Lemma~\ref{lemma:avg}]
This proof is based on Gaussianization (this strategy was suggested to us by B.~Collins). Let $\vY=(Y_1,\ldots,Y_d)$ be a random vector consisting of $d$ independent standard normal random variables, and let $Z=|\vY|$ and $\vX=\vY/Z$; then $\vX$ and $Z$ are independent, and $\vX\sim u$. For $P(x_1,\ldots,x_d) = x_1^{n_1}\cdots x_d^{n_d}$ with $n_1+\ldots+n_d=\ell$, $\EEE P(\vY)$ is, on the one hand, equal to $\EEE P(\vX Z) = \EEE [Z^\ell P(\vX)]=\EEE Z^\ell \, \EEE P(\vX)$ (where the last factor is the quantity we want to compute) and, on the other hand, equal to the product of the $n_j$-th moments of the standard normal distribution; it is known that the $n$-th moment is 0 if $n$ is odd and $(n-1)!!$ if $n$ is even. Thus, $\EEE P(\vX)=0$ for odd $\ell$. 
Since $Z^2\sim\chi^2(d)$, we have that for even $\ell$, $\EEE Z^\ell$ is the $\ell/2$-th moment of the $\chi^2$-distribution with $d$ degrees of freedom, which is known \cite{chi2} to be 
$(d+\ell-2)!!/(d-2)!!$. Thus,
\be
\int_{\SSS(\RRR^d)}u(d\vx)\, x_1^{n_1}\cdots x_d^{n_d} = 
\begin{cases} 
0&\text{if any $n_j$ is odd}\\
\frac{(d-2)!!(n_1-1)!!\cdots (n_d-1)!!}{(d+\ell-2)!!}&\text{if all $n_j$ are even,}
\end{cases}
\ee
which is equivalent to Lemma~\ref{lemma:avg}.
\end{proof}

\begin{proof}[Proof of Theorem~\ref{thm3} in the real case $\X=\RRR^d$]
We first show that the expression \eqref{Tdef} defining $T$ is well defined for any $\psi\in L^2=L^2(\SSS(\RRR^d),u,\CCC)$. In fact, it is well defined for any $\psi \in L^1 =L^1(\SSS(\RRR^d),u,\CCC)$. (Note $L^2\subset L^1$ for a finite measure space such as $(\SSS(\RRR^d),u)$.) To see this, we 
use that $u(d\vx) u_{\SSS(\vx^\perp)}(d\vy)=u(d\vy)u_{\SSS(\vy^\perp)}(d\vx)$; indeed, both equal the unique rotation invariant measure on the subset where $\vx\perp \vy$. Now it follows that for $\psi\in L^1$,
\begin{align}
\label{eqfirst}
\int_{\SSS(\X)}u(d\vx) \Biggl| \int_{\SSS(\vx^\perp)} u_{\SSS(\vx^\perp)}(d\vy) \, \psi(\vy) \Biggr|
&\leq \int_{\SSS(\X)}u(d\vx) \int_{\SSS(\vx^\perp)} u_{\SSS(\vx^\perp)}(d\vy) \, \bigl| \psi(\vy) \bigr|\\
&= \int_{\SSS(\X)}u(d\vy) \int_{\SSS(\vy^\perp)} u_{\SSS(\vy^\perp)}(d\vx) \, \bigl| \psi(\vy) \bigr|\\
&= \int_{\SSS(\X)}u(d\vy) \, \bigl| \psi(\vy) \bigr|\,,
\end{align}
so $T\psi$ is well defined almost everywhere, lies in $L^1$ again, has norm $\|T\psi\|_1 \leq \|\psi\|_1$, and is independent of the choice of representative in the equivalence class that is a vector in $L^1$. Since $L^2 \subset L^1$, the integral formula \eqref{Tdef} is well defined  also for any $L^2$ function. To see that $T\psi\in L^2$ for $\psi\in L^2$, note that by the Cauchy--Schwarz inequality, $|\int \mu(dx)\, f(x)|^2 \leq \int \mu(dx)\, |f(x)|^2$ for any normalized measure $\mu$, so
\be
|T\psi(\vx)|^2 \leq \int_{\SSS(\vx^\perp)} u_{\SSS(\vx^\perp)}(d\vy)\, |\psi(\vy)|^2\,,
\ee
and thus
\begin{align}
\int_{\SSS(\X)}u(d\vx) \, |T\psi(\vx)|^2
&\leq \int_{\SSS(\X)} u(d\vx) \int_{\SSS(\vx^\perp)} u_{\SSS(\vx^\perp)}(d\vy) \, |\psi(\vy)|^2\\
&\leq \int_{\SSS(\X)} u(d\vy) \int_{\SSS(\vy^\perp)} u_{\SSS(\vy^\perp)}(d\vx) \, |\psi(\vy)|^2\\
&\leq \int_{\SSS(\X)} u(d\vy) \, |\psi(\vy)|^2\,,
\end{align}
so $T\psi \in L^2$ for $\psi\in L^2$ and $\|T\psi\|_2\leq \|\psi\|_2$, so $T$ is bounded. To see that it is self-adjoint, note that
\begin{align}
\scp{T\psi}{\chi} 
&= \int_{\SSS(\X)} u(d\vx) \Biggl(\int_{\sphere(\vx^\perp)} u_{\SSS(\vx^\perp)}(d\vy)\, \psi(\vy) \Biggr)^* \chi(\vx)\\
&= \int_{\SSS(\X)} u(d\vx) \int_{\sphere(\vx^\perp)} u_{\SSS(\vx^\perp)}(d\vy)\, \psi^*(\vy) \chi(\vx)\\
&= \int_{\SSS(\X)} u(d\vy) \int_{\sphere(\vy^\perp)} u_{\SSS(\vy^\perp)}(d\vx)\, \psi^*(\vy) \chi(\vx)\\
&=\scp{\psi}{T\chi}\,.\label{eqlast}
\end{align}

Next, observe that $T$ is $O(d)$-invariant,
\be
U(M)TU(M)^{-1}=T \qquad \forall M\in O(d)\,, 
\ee
where $(U(M)\psi)(\vx) = \psi(M\vx)$.
For $\ell=0,1,2,\ldots$, let $\mathscr{A}_\ell$ be the set of all harmonic homogeneous polynomials of degree $\ell$ in $d$ variables; for $\ell=0$ and $1$, $\mathscr{A}_\ell$ is just the set of homogeneous polynomials of degree $\ell$, while for $\ell\geq 2$, the elements are of the form 
\be
P(x_1,\ldots,x_d) = \sum_{i_1\ldots i_\ell=1}^d C_{i_1\ldots i_\ell} \, x_{i_1}\cdots x_{i_\ell}
\ee
with traceless symmetric $C$, i.e.,
\be
\sum_{i=1}^d C_{i_1\ldots i_{\ell-2}ii} = 0
\ee 
for all $i_1,\ldots,i_{\ell-2}\in\{1,\ldots,d\}$. Let $\Hilbert_\ell$ be the set of the restrictions of the $\mathscr{A}_\ell$ functions to $\SSS(\RRR^d)$. The functions in $\Hilbert_\ell$ form the $d$-dimensional analog of the spherical harmonics. It is known (e.g., \cite{VK93,sphericalharmonics}) that the $\Hilbert_\ell$ are irreducible representation spaces of $O(d)$, that they are pairwise inequivalent representations, that they are mutually orthogonal in $\Hilbert$, and that together they span $\Hilbert$ in the $L^2$ norm,
\be
\Hilbert=\bigoplus_{\ell=0}^\infty \Hilbert_\ell\,.
\ee

From this it follows by Schur's lemma that $T$, since it is $O(d)$-invariant, is a multiple of the identity on each $\Hilbert_\ell$. Thus, $T$ has 
pure point spectrum, and each eigenspace must be either one of the $\Hilbert_\ell$ or the sum of several of the $\Hilbert_\ell$. \z{(This observation was made before in \cite{KR11}.)}

To compute the eigenvalue $\tau_\ell$ of $T$ on $\Hilbert_\ell$, it suffices to consider any $P\in\Hilbert_\ell$ and compare the average of $P$ over the equator $\SSS(\RRR^{d-1})=\{\vx\in\SSS(\RRR^d):x_d=0\}$, or $TP(0,0,\ldots,1)$, with the value of $P$ at the north pole, $P(0,\ldots,0,1)=C_{ddd\ldots d}$.

By Lemma~\ref{lemma:avg}, 
the average of $P(\vx)$ with traceless $C$ 
over the equator is 0 for odd $\ell$, while for even $\ell\geq 2$ it is
\begin{align}
\int\limits_{\SSS(\RRR^{d-1})} u_{d-1}(d\vx)\,P(\vx) 
&=\alpha_{\ell,d-1}\sum_{i_1\ldots i_{\ell/2}=1}^{d-1} C_{i_1i_1i_2i_2\ldots i_{\ell/2}i_{\ell/2}}\\
&=-\alpha_{\ell,d-1}\sum_{i_2\ldots i_{\ell/2}=1}^{d-1}C_{ddi_2i_2\ldots i_{\ell/2}i_{\ell/2}}\\
&=(-1)^2\alpha_{\ell,d-1}\sum_{i_3\ldots i_{\ell/2}=1}^{d-1} C_{ddddi_3i_3\ldots i_{\ell/2}i_{\ell/2}}\\
&=(-1)^{\ell/2}\alpha_{\ell,d-1}\,C_{dd\ldots d}\,.
\end{align}
Thus, the eigenvalue of the operator $T$ on $\Hilbert_\ell$ is 
\be
\tau_\ell=
\begin{cases}
  0 & \text{if $\ell$ odd}\\
  (-1)^{\ell/2}\alpha_{\ell,d-1} & \text{if $\ell\geq 2$ even.} 
\end{cases}
\ee
We can now identify the largest absolute eigenvalues. Since, by \eqref{alpha},
\be
\alpha_{\ell+2,d} = \frac{\ell+1}{\ell+d} \alpha_{\ell,d} < \alpha_{\ell,d}\,,
\ee
we have that
\be
\max_{\ell=2,4,6,\ldots} \alpha_{\ell,d} = \alpha_{2,d} = \frac{1}{d}\,.
\ee
Thus, 1 does not occur as an eigenvalue except for constant functions, and the largest absolute non-1 eigenvalue is $\alpha_{2,d-1}=1/(d-1)$.
\end{proof}

\subsection{Proof of Theorem~\ref{thm3} in the Complex Case}

Lemmas~\ref{lemma:invariantC} and \ref{lemma:avgC} provide the complex analogs of Lemmas~\ref{lemma:invariant} and \ref{lemma:avg}. Lemma~\ref{lemma:avgC} is equivalent to Theorem~18 in \cite{CC13}; it is proved there using Gaussianization and here in a different way using Lemma~\ref{lemma:invariantC}.

\begin{lemma}\label{lemma:invariantC}
Let $d\geq 2$ and $\ell,\ell'\in\{0,1,2,\ldots\}$.
Suppose the rank-$(\ell+\ell')$ tensor $A\in (\CCC^d)^{\otimes (\ell+\ell')}$ is symmetric in the first $\ell$ indices and symmetric in the last $\ell'$ indices,
\be\label{sigmaAC}
A_{i_1\ldots i_\ell i'_1\ldots i'_{\ell'}}=A_{i_{\sigma(1)}\ldots i_{\sigma(\ell)} i'_{\sigma'(1)}\ldots i'_{\sigma'(\ell')}}
\qquad \forall \sigma \in S_\ell \:\forall \sigma' \in S_{\ell'}\,,
\ee
and invariant under $U(d)$, acting in the obvious way on the first $\ell$ indices and in the conjugate way on the last $\ell'$ indices,
\be\label{MAC}
\sum_{j_1\ldots j_\ell,j'_1\ldots j'_{\ell'}=1}^d  M_{i_1j_1}\cdots M_{i_\ell j_\ell}\,
\overline{M}_{i'_1j'_1} \cdots \overline{M}_{i'_{\ell} j'_{\ell'}} A_{j_1\ldots j_\ell j'_1\ldots j'_{\ell'}}
 = A_{i_1\ldots i_\ell i'_1 \ldots i'_{\ell'}}
\quad \forall M\in U(d)\,.
\ee
If $\ell\neq \ell'$ then $A=0$, and if $\ell=\ell'$ then $A$ is a multiple of $\tilde{A}$ given by the symmetrization of $\delta_{i_1i'_1}\delta_{i_2i'_2}\cdots \delta_{i_{\ell}i'_\ell}$ in either the primed or the unprimed indices,
\be
\tilde{A}_{i_1\ldots i_\ell i'_1\ldots i'_\ell} = \frac{1}{\ell!} \sum_{\sigma\in S_\ell} \delta_{i_{\sigma(1)}i'_1} \delta_{i_{\sigma(2)} i'_2} \cdots \delta_{i_{\sigma(\ell)} i'_\ell}\,.
\ee
\end{lemma}

\begin{proof}[Proof of Lemma~\ref{lemma:invariantC}]
Also this lemma can be translated into a statement about polynomials. The relevant polynomials to consider are the polynomials $P(z_1,\ldots,z_d,\zbar_1,\ldots,\zbar_d)$ homogeneous of degree $\ell$ in $z$ and degree $\ell'$ in $\zbar$; they can be thought of as complex polynomials in $2d$ complex variables, with the conjugates of $z_1,\ldots,z_d$ inserted as the last $d$ variables; they can be written as
\be\label{PAC}
P(z_1,\ldots,z_d,\zbar_1,\ldots,\zbar_d) = \sum_{i_1,\ldots,i_\ell,i'_1,\ldots,i'_{\ell'}=1}^d A_{i_1\ldots i_\ell,i'_1\ldots i'_{\ell'}}\, z_{i_1}\cdots z_{i_\ell}\, \zbar_{i'_1} \cdots \zbar_{i'_{\ell'}}\,.
\ee
The pair $(\ell,\ell')$ is called the \emph{bi-degree}\footnote{Again, we include the possibility $P=0$.} of $P$.
Lemma~\ref{lemma:invariantC} can then be paraphrased as:

\textit{Suppose the bi-homogeneous polynomial $P(z_1,\ldots,z_d,\zbar_1,\ldots,\zbar_d)$ of bi-degree $(\ell,\ell')$ is $U(d)$-invariant. If $\ell\neq \ell'$ then $P=0$, and if $\ell=\ell'$ then $P$ is a multiple of $(|z_1|^2+ \ldots + |z_d|^2)^{\ell}$.}

Considering \eqref{MAC} for $M=e^{i\theta}I$ with $\theta\in \RRR$ and $I$ the identity matrix, we obtain that $e^{i(\ell-\ell')\theta} A = A$, so for $\ell\neq\ell'$ we have that $A=0$ (and $P=0$).
Now assume $\ell=\ell'$. Since $P$ is $U(d)$-invariant, its restriction to $\SSS(\CCC^d)$ must be constant. Since $P(\vz,\vzbar)$ is real-homogeneous of degree $2\ell$, it must be of the form $c|\vz|^{2\ell}$, where $c$ is a complex constant and $|\vz|=\sqrt{|z_1|^2+\ldots+|z_d|^2}$.
\end{proof}

\begin{lemma}\label{lemma:avgC}
Suppose $d\geq 2$ and $\ell,\ell'\in\{0,1,2,\ldots\}$. Let $P(\vz,\vzbar)$ be a bi-homogeneous polynomial of bi-degree $(\ell,\ell')$, 
\be
P(\vz,\vzbar) = \sum_{i_1\ldots i_\ell, i'_1\ldots i'_{\ell'}=1}^d C_{i_1\ldots i_\ell i'_1 \ldots i'_{\ell'}} \, 
z_{i_1}\cdots z_{i_\ell} \, \zbar_{i'_1} \cdots \zbar_{i'_{\ell'}}
\ee
with a complex tensor $C$ that is symmetric in the primed and in the unprimed indices.
Then the average of $P$ over the unit sphere is
\be\label{PavgC}
\int\limits_{\SSS(\CCC^d)} \!\! u(d\vz)\, P(\vz,\vzbar) = 
\begin{cases}
0 & \text{if $\ell\neq \ell'$}\\[3mm]
\beta_{\ell,d} \sum\limits_{i_1\ldots i_{\ell}=1}^d C_{i_1\ldots i_{\ell}i_1\ldots i_{\ell}} & \text{if $\ell=\ell'\geq 1$}
\end{cases}
\ee
with 
\be\label{beta}
\beta_{\ell,d}= \binom{\ell+d-1}{\ell}^{\!\!-1}\,.
\ee
\end{lemma}

\begin{proof}[Proof of Lemma~\ref{lemma:avgC}]
By linearity, the average of $P$ must be
\be
\sum_{i_1\ldots i_{\ell},i'_1\ldots i'_{\ell'}=1}^d A_{i_1\ldots i_{\ell},i'_1\ldots i'_{\ell'}} \, C_{i_1\ldots i_{\ell},i'_1\ldots i'_{\ell'}}
\ee
with
\be
A_{i_1\ldots i_{\ell},i'_1\ldots i'_{\ell'}} = \int\limits_{\SSS(\CCC^d)} u(d\vz) \, z_{i_1}\cdots z_{i_{\ell}}\, \zbar_{i'_1} \cdots \zbar_{i'_{\ell'}}
\ee
The tensor $A$ is symmetric in the first $\ell$ variables and symmetric in the last $\ell'$ variables, and $U(d)$-invariant in the sense of \eqref{MAC}. Lemma~\ref{lemma:invariantC} now yields $A=\beta_{\ell,d}\tilde{A}$ and thus \eqref{PavgC} except for the value of the constant $\beta_{\ell,d}$. 

To compute $\beta_{\ell,d}$, note that $\tilde{A}_{1\ldots 11\ldots 1}=1$, so
\begin{align}
\beta_{\ell,d}&=A_{1\ldots 11\ldots 1}\\ 
&= \int\limits_{\SSS(\CCC^d)} \!\! u(d\vz) \, |z_1|^{2\ell}\\
&= \int\limits_{\SSS(\RRR^{2d})} \!\! u(d\vx)\, (x_1^2+x_2^2)^\ell\label{x12plusx22}\\
&= \sum_{k=0}^\ell \binom{\ell}{k} \int\limits_{\SSS(\RRR^{2d})} \!\! u(d\vx)\,  x_1^{2k} x_2^{2(\ell-k)} \\
\intertext{[using \eqref{Pavg} for $P(x_1,\ldots,x_{2d})=x_1^{2k}x_2^{2(\ell-k)}$, which has $C_{i_1\ldots i_{2d}} = \binom{2\ell}{2k}^{-1}$ if $2k$ of the $i_j$ are 1 and the others are $2$, and $C_{i_1\ldots i_{2d}}=0$ otherwise, so the
sum in \eqref{Pavg} has $\binom{\ell}{k}$ nonzero terms]}
&= \alpha_{2\ell,2d} \sum_{k=0}^\ell \binom{\ell}{k}^{\! 2} \binom{2\ell}{2k}^{\!-1} \\
&= \alpha_{2\ell,2d} 4^\ell \binom{2\ell}{\ell}^{\!-1}\,,\label{lastA1a}
\end{align}
where the last step can be obtained either from Gauss's theorem about the hypergeometric function ${}_2F_1(a,b;c;z)$ at $z=1$ \cite{hypergeometricfunction}, or using Zeilberger's algorithm \cite{Z91}.
One easily verifies that \eqref{lastA1a} is equal to \eqref{beta}.

Alternatively, starting from \eqref{x12plusx22}, we can evaluate this integral by noting that for $\vx=(x_1,\ldots,x_D)=(\vx^{(1)},\vx^{(2)})\in\RRR^D$ with $\vx^{(1)}=(x_1,\ldots,x_n)$ and $\vx^{(2)}=(x_{n+1},\ldots,x_{n+m})$ such that $n+m=D$, we have that for the $D$-dimensional volume measure,
\begin{align}
d\vx &= d\vx^{(1)}\, d\vx^{(2)}\\
&=r_1^{n-1}\, d\Omega^{(1)} \, dr_1 \, r_2^{m-1} \, d\Omega^{(2)}\, dr_2\\
&=\rho\, r_1^{n-1}\, r_2^{m-1}\, d\Omega^{(1)} \, d\Omega^{(2)}\, d\theta \, d\rho\,,
\end{align}
where \z{$r_i=|\vx^{(i)}|$, $d\Omega^{(i)}$ is the solid angle for $\vx^{(i)}/r_i$,} $\rho=|\vx|=\sqrt{r_1^2+r_2^2}$, $\cos \theta= r_1/\rho$ ($0\leq \theta\leq \pi/2$), so $r_2=\rho \, \sin\theta$. This yields, for $n=2$ and $D=2d$,
\begin{align}
\beta_{\ell,d}&= \int\limits_{\SSS(\RRR^{2d})} \!\! u(d\vx) \, (x_1^2+x_2^2)^\ell\\
&=\frac{1}{|\SSS(\RRR^{2d})|} \int_0^{\pi/2} \!\! d\theta\, \cos^{2\ell}\theta \, \cos\theta \, \sin^{2d-3}\theta \, 2\pi|\SSS(\RRR^{2d-2})|\\ 
&= 2\pi\frac{(2d-2)!!\,g(2d-2)}{(2d-4)!!\,g(2d)} \int_0^{\pi/2} d\theta\, \sin^{2d-3}\theta\, \cos^{2\ell+1} \theta \\
&= (2d-2) \frac{(2\ell)!! (2d-4)!!}{(2d+2\ell-2)!!}\\
&= \frac{(2\ell)!! (2d-2)!!}{(2d+2\ell-2)!!}\label{lastA1c}
\end{align}
using
\be
\int_0^{\pi/2} d\theta \, \sin^p \theta \, \cos^q \theta = \frac{(q-1)!! (p-1)!!}{(p+q)!!}
\ee
for odd $q$. One easily verifies that \eqref{lastA1c} is equal to \eqref{beta}.
\end{proof}

\bigskip

\begin{proof}[Proof of Theorem~\ref{thm3} in the complex case $\X=\CCC^d$]
By the same reasoning as in the real case, involving \eqref{eqfirst}--\eqref{eqlast}, $T$ must be self-adjoint and bounded. Clearly, it is $U(d)$-invariant.
For $\ell,\ell'\in\{0,1,2,\ldots\}$, let $\mathscr{A}_{\ell\ell'}$ be the set of all harmonic bi-homogeneous polynomials $P(z_1,\ldots,z_d,\zbar_1,\ldots,\zbar_d)$ of bi-degree $(\ell,\ell')$; for $\ell,\ell'\geq 1$, they are of the form 
\be
P(\vz,\vzbar) = \sum_{i_1\ldots i_\ell,i'_1\ldots i'_{\ell'}=1}^d C_{i_1\ldots i_\ell i'_1 \ldots i'_{\ell'}} \, z_{i_1}\cdots z_{i_\ell}\,\zbar_{i_1}\cdots \zbar_{i'_{\ell'}}
\ee
with a tensor $C$ that is symmetric in the sense of \eqref{sigmaAC} and traceless in the sense that
\be
\sum_{i=1}^d C_{i_1\ldots i_{\ell-1}i, i'_1\ldots i'_{\ell'-1}i} = 0
\ee 
for all $i_1,\ldots,i_{\ell-1},i'_1,\ldots,i'_{\ell'-1}\in\{1,\ldots,d\}$. Let $\Hilbert_{\ell\ell'}$ be the set of the restrictions of the $\mathscr{A}_{\ell\ell'}$ functions to $\SSS(\CCC^d)$. The functions in $\Hilbert_{\ell\ell'}$ form the complex analog of the spherical harmonics. It is known (e.g., \cite[p.~296]{VK93}) that the $\Hilbert_{\ell\ell'}$ are irreducible representation spaces of $U(d)$, that they are pairwise inequivalent representations \cite[p.~296]{VK93}, that they are mutually orthogonal \cite[p.~293]{VK93}, and that together they span $\Hilbert$ in the $L^2$ norm \cite[p.~294]{VK93},
\be
\Hilbert=\bigoplus_{\ell,\ell'=0}^\infty \Hilbert_{\ell\ell'}\,.
\ee
By the same reasoning as in the real case, each eigenspace of $T$ must be either one $\Hilbert_{\ell\ell'}$ or the sum of several ones.

To compute the eigenvalue $\tau_{\ell\ell'}$ of $T$ on $\Hilbert_{\ell\ell'}$, we consider any $P\in\Hilbert_{\ell\ell'}$ and compare the average of $P$ over the equator $\SSS(\CCC^{d-1})=\{\vz\in\SSS(\CCC^d):z_d=0\}$, or $TP(0,0,\ldots,1)$, with the value of $P$ at the north pole, $P(0,\ldots,0,1)=C_{d\ldots dd\ldots d}$.

By Lemma~\ref{lemma:avgC}, 
the average of $P$ over the equator is 0 for $\ell\neq \ell'$, while for $\ell=\ell'\geq 1$ it is
\begin{align}
\int\limits_{\SSS(\CCC^{d-1})} u_{d-1}(d\vz)\,P(\vz,\vzbar) 
&=\beta_{\ell,d-1}\sum_{i_1\ldots i_{\ell}=1}^{d-1} C_{i_1\ldots i_\ell,i_1\ldots i_{\ell}}\\
&=-\beta_{\ell,d-1}\sum_{i_1\ldots i_{\ell-1}=1}^{d-1}C_{i_1\ldots i_{\ell-1}d,i_1\ldots i_{\ell-1}d}\\
&=(-1)^{\ell}\beta_{\ell,d-1}\,C_{dd\ldots d}\,.
\end{align}
Thus, the eigenvalue of the operator $T$ on $\Hilbert_{\ell\ell'}$ is 
\be
\tau_{\ell\ell'}=
\begin{cases}
  0 & \text{if $\ell\neq \ell'$}\\
  (-1)^{\ell}\beta_{\ell,d-1} & \text{if $\ell=\ell'\geq 1$.} 
\end{cases}
\ee
For $\ell=\ell'=0$, of course, $\Hilbert_{\ell\ell'}=\Hilbert_{00}$ consists of the constant functions, and the eigenvalue is $\tau_{00}=1$.
We can now identify the largest absolute eigenvalues. Since, by \eqref{beta},
\be
\beta_{\ell+1,d} = \frac{\ell+1}{\ell+d} \beta_{\ell,d} < \beta_{\ell,d}\,,
\ee
we have that
\be
\max_{\ell=1,2,3,\ldots} \beta_{\ell,d} = \beta_{1,d} = \frac{1}{d}\,.
\ee
Thus, 1 does not occur as an eigenvalue except for constant functions, and the largest absolute non-1 eigenvalue is $\beta_{1,d-1}=1/(d-1)$.
\end{proof}

\section{Application}
\label{sec:appl}

A physical application of our results, described in detail in \cite{GLMTZ14}, concerns quantum statistical mechanics, in particular the distribution of the wave function in thermal equilibrium. 

\subsection{Setup}

Consider any quantum system $S$ weakly coupled to another system $B$ with a large (but finite) number of particles; $B$ is called the ``heat bath.'' Suppose that the composite system $S\cup B$ is isolated, with Hilbert space
\be
\Hilbert=\Hilbert_S \otimes \Hilbert_B
\ee
and the Hamiltonian
\be
H = H_S \otimes I_B + I_S\otimes H_B + H_{SB}\,,
\ee
where $I$ denotes the identity operator, and the interaction term $H_{SB}$ is assumed to be small and will be neglected for much of the reasoning. In correspondence to the physical assumption that $S\cup B$ is constrained to a finite volume of 3-space, we assume that $H$ has pure point spectrum. Consider an energy interval $[E,E+\delta E]$ that is small on the macroscopic scale but large enough to contain many eigenvalues of $H$. Let the ``micro-canonical'' subspace $\Hilbert_{mc}$ of $\Hilbert$ be the spectral subspace corresponding to $[E,E+\delta E]$, i.e., $\Hilbert_{mc}$ is spanned by the eigenvectors of $H$ with eigenvalues between $E$ and $E+\delta E$, and suppose that $S\cup B$ is in a pure state $\psi$ in $\Hilbert_{mc}$. Without loss of generality, $\Hilbert_{mc}$, $\Hilbert_S$, and $\Hilbert_B$ can be taken to have finite dimension, while $\dim\Hilbert_{mc}$ and $\dim\Hilbert_B$ should be large (like $\exp(10^{10})$). Most wave functions $\psi\in\SSS(\Hilbert_{mc})$ (``most'' relative to $u_{\SSS(\Hilbert_{mc})}$) represent states of thermal equilibrium. According to a fact known as ``canonical typicality'' \cite{GMM04,Gold2,PSW06}, most $\psi\in\SSS(\Hilbert_{mc})$ are such that, for $\dim \Hilbert_S\ll \dim\Hilbert_{mc}$, 
\be
\rho_S^\psi \approx \rho_\beta\,,
\ee
where $\rho_S^\psi$ denotes the reduced density matrix of $S$,
\be
\rho_S^\psi := \tr_B |\psi\rangle\langle\psi|\,,
\ee
and $\rho_\beta$ the ``canonical'' density matrix associated with inverse temperature $\beta=1/kT$ ($k$ = Boltzmann's constant, $T$ = temperature),
\be
\rho_\beta := \frac{1}{Z} e^{-\beta H}
\ee
with $Z=\tr e^{-\beta H}$; the value of $\beta$ is determined by $E$ and the sizes of $S$ and $B$.

\subsection{Conditional Wave Function}

As first pointed out in \cite{Gold1}, it is also true for most $\psi\in\SSS(\Hilbert_{mc})$ that the ``conditional wave function'' $\psi_S$ of system $S$ (see below) has a probability distribution that depends only on $H_S$ and $\beta$ (and thus does not depend on $H_B$, $H_{SB}$ if small enough, or on the details of $\psi$), called the ``thermal equilibrium distribution of $\psi_S$.'' This distribution is $GAP(\rho_\beta)$, the Gaussian Adjusted Projected measure with covariance operator $\rho_\beta$ \cite{Gold1}. The mathematical proof \cite{GLMTZ14} of this statement  of ``GAP typicality'' is where Theorems~\ref{thm1} and \ref{thm2} are useful.

To explain this further, we first elucidate the concept of ``conditional wave function.'' Given an orthonormal basis $\{b_1,\ldots,b_d\}$ of $\Hilbert_B$ and a vector $\psi\in\SSS(\Hilbert)$, the conditional wave function $\psi_S$ is a random vector in $\SSS(\Hilbert_S)$, obtained from $\psi$ by means of the partial inner product,
\be
\psi_S = \frac{1}{\mathcal{N}} \scp{b_J}{\psi}_B\,,
\ee
with a random basis vector $b_J$, chosen with the Born-rule distribution
\be
\PPP(J=j) = \Bigl\| \scp{b_j}{\psi}_B \Bigr\|_{\Hilbert_S}\,.
\ee
($\mathcal{N}$ is a normalizing factor, and the partial inner product $\phi=\scp{b}{\psi}_B$ is defined by the property $\scp{\chi}{\phi}_{\Hilbert_S}=\scp{\chi\otimes b}{\psi}_{\Hilbert}$.) Usually, $\psi_S$ depends on $\psi$ as well as on the basis $\{b_1,\ldots, b_d\}$; however, in the special situation of thermal equilibrium, the distribution does not depend on the choice of basis, nor (as already mentioned) on $\psi$ (except through $H_S$ and $\beta$).

A key to proving GAP typicality is this statement: \textit{If $\{b_1,\ldots,b_d\}$ is a random orthonormal basis of $\Hilbert_B$ then, for every $\psi\in\SSS(\Hilbert)$, the distribution of $\psi_S$ is close to $GAP(\rho_S^\psi)$ with probability near 1.} To prove this statement, two things are relevant: First, that when, for fixed $\psi$, the distribution of $\psi_S$ on $\SSS(\Hilbert_S)$ is \emph{averaged} over all orthonormal bases $\{b_1,\ldots,b_d\}$, the result is $GAP(\rho_S^\psi)$. And second, the result of the present paper. That is because the distribution of $\psi_S$ is actually of the form
\be\label{quantity}
\frac{1}{d}\sum_{i=1}^d \varphi(b_i)\,,
\ee
where $\varphi$ is a function on $\SSS(\Hilbert_B)$ that yields measures on $\SSS(\Hilbert_S)$ as values. Theorem~\ref{thm1} shows that the measure \eqref{quantity} will, with high probability, be close to its average 
\be
\int_{\SSS(\Hilbert_B)} u(d\vx) \, \varphi(\vx)=GAP(\rho_S^\psi)\,,
\ee
as claimed.

\bigskip

\noindent\textit{Acknowledgments.} 
We thank J\'ozsef Beck, Beno\^it Collins, Roe Goodman, Neil Sloane, and Doron Zeilberger for helpful discussions. 
We acknowledge support from the National Science Foundation [grant DMS-0504504 to S.G.; DMR 08-02120 to J.L.L.], the Air Force Office of Scientific Research [grant AF-FA 49620-01-0154 to J.L.L.], the John Templeton Foundation [grant 37433 to S.G.\ and R.T.], the European Cooperation in Science and Technology [COST action MP1006 to N.Z.], and Istituto Nazionale di Fisica Nucleare [to N.Z.].

\end{document}